\documentclass[12pt]{article}
\usepackage{epsf}

\usepackage{amsthm}
\usepackage{amssymb}
\usepackage{amsmath}

\usepackage[top=1.5cm,bottom=1.5cm,left=2.5cm,right=2.5cm,includehead,includefoot
           ]{geometry}

\makeatletter

\numberwithin{equation}{section}
\newtheorem{theorem}{Theorem}[section]
\newtheorem{lemma}[theorem]{Lemma}

\title{Characterizations of centralizable mappings on algebras of locally measurable operators}
\author{\begin{tabular}{c} Guangyu An$^{1}$, Jun He$^{2}$\footnote{Corresponding author.
E-mail address: hejun$_{-}$12@163.com} and Jiankui Li$^{3}$
\\{\small\it  $^{1}$Department of Mathematics, Shaanxi University of Science and Technology}\\
{\small\it Xi'an 710021, China}
\\{\small\it $^{2}$Department of Mathematics, Anhui Polytechnic University}\\
{\small\it Wuhu 241000, China}
\\{\small\it $^{3}$Department of Mathematics, East China University of
Science and Technology}\\
{\small\it Shanghai 200237, China}
\end{tabular}}
\date{}
\begin{document}
\maketitle \abstract
A linear mapping $\phi$ from an algebra $\mathcal{A}$ into its bimodule $\mathcal M$ is called
a centralizable mapping at $G\in\mathcal{A}$ if $\phi(AB)=\phi(A)B=A\phi(B)$
for each $A$ and $B$ in $\mathcal{A}$ with $AB=G$. In this paper, we
prove that if $\mathcal M$ is a von Neumann algebra without direct summands of
type $\mathrm{I}_1$ and type $\mathrm{II}$, $\mathcal A$ is a $*$-subalgebra with
$\mathcal M\subseteq\mathcal A\subseteq LS(\mathcal{M})$ and $G$ is a fixed element in $\mathcal A$, then
every continuous
(with respect to the local measure topology $t(\mathcal M)$) centralizable mapping at $G$
from $\mathcal A$ into $\mathcal M$ is a centralizer.

\

\

\
\section{Introduction}\

Let $\mathcal{A}$ be an associative algebra
over the complex field $\mathbb{C}$, $\mathcal{M}$ be an $\mathcal{A}$-bimodule,
and $L(\mathcal A,\mathcal M)$ be the set of all linear mappings from $\mathcal{A}$ into $\mathcal{M}$.
If $\mathcal A=\mathcal M$, then denote $L(\mathcal A,\mathcal A)$ by $L(\mathcal A)$.
A linear mapping $\phi$ in $L(\mathcal A,\mathcal M)$ is
called a \emph{centralizer} if $\phi(AB)=\phi(A)B=A\phi(B)$ for each $A$ and $B$ in $\mathcal{A}$.
In particular, if $\mathcal A$ is a unital algebra with a unit element $I$, then
$\phi$ is a centralizer if and only if $\phi(A)=\phi(I)A=A\phi(I)$ for every $A$
in $\mathcal A$.

Let $G$ be a fixed element in $\mathcal A$. A linear mapping $\phi$ in $L(\mathcal A,\mathcal M)$ is called
a \emph{centralizable mapping at $G$} if $\phi(AB)=\phi(A)B=A\phi(B)$
for each $A$ and $B$ in $\mathcal{A}$ with $AB=G$.
Moreover, we say that $G$ is a \emph{full-centralizable point of $L(\mathcal A,\mathcal M)$} if
every centralizable mapping at $G$ from $\mathcal A$ into $\mathcal M$ is a centralizer.

Suppose that $\mathcal R$ is a prime ring with a nontrival idempotent,
in \cite{501}, M. Bre\v{s}ar shows that zero is a full-centralizable point of $L(\mathcal R)$;
and in \cite{503}, X. Qi shows that every nontrival
idempotent in $\mathcal R$ is a full-centralizable point of $L(\mathcal R)$.
In \cite{519}, W. Xu, R. An and J. Hou prove that if $\mathcal H$
is a Hilbert space with $\mathrm{dim}\mathcal H\geq2$, then every element $G$ in $B(\mathcal H)$ is
a full-centralizable point of $L(B(\mathcal H))$.
In \cite{Hejun}, J. He, J. Li and W. Qian prove that if $\mathcal M$ is a von Neumann algebra,
then every element $G$ in $\mathcal M$ is
a full-centralizable point of $L(\mathcal M)$.

Let $\mathcal H$ be a complex Hilbert space
and $B(\mathcal H)$ be the algebra of all bounded linear operators on $\mathcal H$.
Suppose that $\mathcal{M}$ is a von Neumann algebra on
$\mathcal{H}$ and
$\mathcal Z(\mathcal M)=\mathcal M\cap\mathcal M'$
is the center of $\mathcal M$, where
$\mathcal M'=\{A\in B(\mathcal H):AB=BA~\mathrm{for~every}~B~\mathrm{in}~\mathcal M\}.$
Denote by
$\mathcal P(\mathcal M)=\{P\in\mathcal M:P=P^*=P^2\}$
the lattice of all projections in $\mathcal M$ and by $\mathcal P_{fin}(\mathcal M)$
the set of all finite projections in $\mathcal M$.

Let $T$ be a closed densely defined linear operator on $\mathcal H$
with the domain $\mathcal D(T)$, where $\mathcal D(T)$ is a linear subspace of $\mathcal H$.
$T$ is said to be \emph{affiliated} with $\mathcal M$,
denote by $T\eta\mathcal M$, if
$U^*TU=T$ for every unitary element $U$ in $\mathcal{M}'$.

A linear operator $T$ affiliated with $\mathcal M$
is said to be \emph{measurable} with respect to
$\mathcal M$, if there exists a sequence $\{P_n\}_{n=1}^{\infty}\subset\mathcal P(\mathcal M)$ such that $P_n\uparrow1$,
$P_n(\mathcal H)\subset\mathcal D(T)$ and $P_n^{\bot}=I-P_n\in \mathcal P_{fin}(\mathcal M)$ for every $n\in\mathbb{N}$,
where $\mathbb{N}$ is the set of all natural numbers. Denote by $S(\mathcal M)$ the set of all measurable operators
affiliated with the von Neumann algebra $\mathcal M$.

A linear operator $T$ affiliated with $\mathcal M$
 is said to be \emph{locally measurable} with respect to
$\mathcal M$, if there exists a sequence $\{Z_n\}_{n=1}^{\infty}\subset\mathcal P(\mathcal Z(\mathcal M))$
such that $Z_n\uparrow I$ and $Z_nT\in S(\mathcal M)$ for every $n\in\mathbb{N}$.
Denote by $LS(\mathcal M)$ the set of all locally measurable operators
affiliated with the von Neumann algebra $\mathcal M$.

In \cite{M. Muratov}, M. Muratov and V. Chilin prove that $S(\mathcal M)$ and $LS(\mathcal M)$
are both unital $*$-algebras and $\mathcal M\subset S(\mathcal M)\subset LS(\mathcal M)$;
the authors also show that if $\mathcal M$ is a finite von Neumann algebra
or $\mathrm{dim}(\mathcal Z(\mathcal M))<\infty$, then
$S(\mathcal M)=LS(\mathcal M)$; if $\mathcal M$ is a type III von Neumann algebra and
$\mathrm{dim}(\mathcal Z(\mathcal M))=\infty$, then $S(\mathcal M)=\mathcal M$ and
$LS(\mathcal M)\neq\mathcal M$.

In \cite{Segal}, I. Segal shows that the algebraic and topological properties
of the measurable operators algebra $S(\mathcal M)$ are similar to the von Neumann algebra $\mathcal M$.
If $\mathcal M$ is a commutative von Neumann algebra,
then $\mathcal M$ is $*$-isomorphic to the algebra
$L^\infty(\Omega,\Sigma,\mu)$ of all essentially bounded measurable
complex functions on a measure space $(\Omega,\Sigma,\mu)$; and
$S(\mathcal M)$ is $*$-isomorphic to the algebra $L^0(\Omega,\Sigma,\mu)$
of all measurable almost everywhere finite complex-valued functions on
$(\Omega,\Sigma,\mu)$. In \cite{Ber1}, A. Ber, V. Chilin and F. Sukochev show that there exists a derivation on $L^0(0,1)$
is not an inner derivation, and the derivation is discontinuous in the measure topology.
This result means that the properties of derivations on $S(\mathcal M)$
are different from the derivations
on $\mathcal M$.

So far, there are no papers on the study of the centralizable mappings on algebras of locally measurable operators.
This paper is organized as follows.

In Section 2, we suppose that $\mathcal M$ is a von Neumann algebra and
recall the definition of local measurable topology $t(\mathcal M)$ on $LS(\mathcal M)$.

Let $\mathcal A$ be a subalgebra of $LS(\mathcal M)$. Denote by $L_{t(\mathcal M)}(\mathcal A, LS(\mathcal{M}))$ the set of all continuous
linear mappings with respect to the local measure topology $t(\mathcal M)$ from $\mathcal A$ into $LS(\mathcal{M})$.
Suppose that $G$ is a fixed element in $\mathcal A$, we say that $G$ is a full-centralizable point
of $L_{t(\mathcal M)}(\mathcal A,LS(\mathcal M))$ if
every continuous (with respect to the
local measure topology $t(\mathcal M)$) centralizable mapping at $G$
from $\mathcal A$ into $\mathcal M$ is a centralizer.

In Section 3, we show that if $\mathcal M$ is a von Neumann algebra without direct summands of
type $\mathrm{I}_1$ and type $\mathrm{II}$, and $\mathcal A$ is a $*$-subalgebra with
$\mathcal M\subseteq\mathcal A\subseteq LS(\mathcal{M})$, then every element
$G$ in $\mathcal A$ is a full-centralizable point of $L_{t(\mathcal M)}(\mathcal A,LS(\mathcal M))$.

\section{Preliminaries}\

Let $\mathcal H$ be a complex Hilbert space
and $\mathcal{M}$ be a von Neumann algebra on
$\mathcal{H}$.
Suppose that $T$ is a closed operator with a dense domain $\mathcal D(T)$ in $\mathcal H$.
Let $T=U|T|$ be the polar decomposition of $T$, where $|T|=(T^*T)^{\frac{1}{2}}$ and
$U$ is a partial isometry in $B(\mathcal H)$. Denote by $l(T)=UU^*$ the
\emph{left support} of $T$ and by $r(T)=U^*U$ the
\emph{right support} of $T$, clearly, $l(T)\sim u(T)$.
In \cite{M. Muratov},  M. Muratov and V. Chilin show that $T\in S(\mathcal M)$ (resp. $T\in LS(\mathcal M)$)
if and only if $|T|\in S(\mathcal M)$ (resp. $|T|\in LS(\mathcal M)$)
and $U\in\mathcal M$.

In the following, we recall the definition of the local
measure topology.
Let $\mathcal M$ be a commutative von Neumann algebra,
in \cite{M. Takesaki}, M. Takesaki proves that there exists a $*$-isomorphism from
$\mathcal M$ onto the $*$-algebra
$L^\infty(\Omega,\Sigma,\mu)$, where $\mu$ is a measure
satisfying the direct sum property. The direct sum property
means that the Boolean algebra of all projections
in $L^\infty(\Omega,\Sigma,\mu)$ is total order, and for every
non-zero projection $p$ in $\mathcal M$, there exists a non-zero projection $q\leq p$
with $\mu(q)<\infty$. Consider $LS(\mathcal M)=S(\mathcal M)=L^0(\Omega,\Sigma,\mu)$
of all measurable almost everywhere finite complex-valued functions on $(\Omega,\Sigma,\mu)$.
Define the local measure topology $t(L^\infty(\Omega))$ on $L^0(\Omega,\Sigma,\mu)$, that is,
the Hausdorff vector topology, whose base of neighborhoods of zero is given by
\begin{align*}
W(B,\varepsilon,\delta)=&\{f\in L^0(\Omega,\Sigma,\mu): \mathrm{there~exists~a~set}~E\in\Sigma~\mathrm{such~that}\\
&E\subset B, \mu(B\backslash E)\leq\delta,f_{\chi_{E}}\in L^\infty(\Omega,\Sigma,\mu),\|f_{\chi_{E}}\|_{L^\infty(\Omega,\Sigma,\mu)}\leq\varepsilon\},
\end{align*}
where $\varepsilon,\delta>0,B\in\Sigma,\mu(B)<\infty$ and $\chi_{E}(\omega)=1$ when $\omega\in E$, $\chi_{E}(\omega)=0$ when $\omega\notin E$.
Suppose that $\{f_\alpha\}\subset L^0(\Omega,\Sigma,\mu)$ and $f\in L^0(\Omega,\Sigma,\mu)$, if
$f_\alpha\chi_B\rightarrow f\chi_B$ in the measure $\mu$ for every $B\in\Sigma$ with $\mu(B)<\infty$,
then we denote by
$f_\alpha\xrightarrow{t(L^\infty(\Omega))}f$. In \cite{F. Yeadon}, Yeadon show the topology
$t(L^\infty(\Omega))$ dose not change if the measure $\mu$ is replaced with an equivalent measure.

If $\mathcal M$ is an arbitrary von Neumann algebra
and $\mathcal Z(\mathcal M)$
is the center of $\mathcal M$, then
there exists a $*$-isomorphism $\varphi$ from
$\mathcal Z(\mathcal M)$ onto the $*$-algebra
$L^\infty(\Omega,\Sigma,\mu)$, where $\mu$ is a measure
satisfying the direct sum property.
Denote by $L^+(\Omega,\Sigma,\mu)$
the set of all measurable real-valued positive functions on $(\Omega,\Sigma,\mu)$.
In \cite{Segal}, Segal shows that there exists a mapping
$\Delta$ from $\mathcal P(\mathcal M)$ into $L^+(\Omega,\Sigma,\mu)$ satisfying the following conditions:\\
($\mathbb{D}_{1}$) $\Delta(P)\in L_+^0(\Omega,\Sigma,\mu)$ if and only if $P\in \mathcal P_{fin}(\mathcal M)$;\\
($\mathbb{D}_{2}$) $\Delta(P\vee Q)=\Delta(P)+\Delta(Q)$ if $PQ=0$;\\
($\mathbb{D}_{3}$) $\Delta(U^*U)=\Delta(UU^*)$ for every partial isometry $U\in\mathcal M$;\\
($\mathbb{D}_{4}$) $\Delta(ZP)=\varphi(Z)\Delta(P)$ for every $Z\in\mathcal P(\mathcal Z(\mathcal M))$
and every $P\in\mathcal P(\mathcal M)$;\\
($\mathbb{D}_{5}$) if $P_\alpha,P\in\mathcal P(\mathcal M),\alpha\in\Gamma$ and $P_\alpha\uparrow P$, then $\Delta(P)=\mathrm{sup}_{\alpha\in\Gamma}\Delta(P_\alpha)$.\\
In addition, $\Delta$ is called a \emph{dimension function} on $\mathcal P(\mathcal M)$
and $\Delta$ also satisfies the following two conditions:\\
($\mathbb{D}_{6}$) if $\{P_n\}_{n=1}^\infty\subset\mathcal P(\mathcal M)$,
then $\Delta(\mathrm{sup}_{n\geq1}P_n)\leq\sum_{n=1}^{\infty}\Delta(P_n)$; moreover, if $P_nP_m=0$
when $n\neq m$, then $\Delta(\mathrm{sup}_{n\geq1}P_n)=\sum_{n=1}^{\infty}\Delta(P_n)$;\\
($\mathbb{D}_{7}$) if $\{P_n\}_{n=1}^\infty\subset\mathcal P(\mathcal M)$ and $P_n\downarrow0$, then $\Delta(P_n)\rightarrow0$ almost everywhere.

For arbitrary scalars $\varepsilon,\gamma>0$ and a set $B\in\Sigma$, $\mu(B)<\infty$, we let
\begin{align*}
V(B,\varepsilon,\gamma)=&\{T\in LS(\mathcal M):~\mathrm{there~exist}~P\in\mathcal P(\mathcal M)~\mathrm{and}~Z\in\mathcal P(Z(\mathcal M)~\mathrm{such~that}\\&TP\in\mathcal M, \|TP\|_{\mathcal M}\leq\varepsilon,\varphi(Z^\bot)\in W(B,\varepsilon,\gamma)~\mathrm{and}~\Delta(ZP^\bot)\leq\varepsilon\varphi(Z)\},
\end{align*}
where $\|\cdot\|_{\mathcal M}$ is the $C^*$-norm on $\mathcal M$.
In \cite{F. Yeadon}, Yeadon shows that the system of sets
$$\{T+V(B,\varepsilon,\gamma):T\in LS(\mathcal M),\varepsilon,\gamma>0,B\in\Sigma~\mathrm{and}~\mu(B)<\infty\}$$
defines a Hausdorff vector topology $t(\mathcal M)$ on $LS(\mathcal M)$ and the sets $$\{T+V(B,\varepsilon,\gamma),\varepsilon,\gamma>0,B\in\Sigma~\mathrm{and}~\mu(B)<\infty\}$$
form a neighborhood base of a local measurable operator $x$ in $LS(\mathcal M)$.
In \cite{F. Yeadon}, Yeadon also proves that $(LS(\mathcal M),t(\mathcal M))$ is
a complete topological $*$-algebra, and the topology $t(\mathcal M)$
does not depend on the choices of dimension function $\Delta$ and $*$-isomorphism $\varphi$.
The topology $t(\mathcal M)$ on $LS(\mathcal M)$ is called the \emph{local measure topology}.
Moreover, if $\mathcal M=B(\mathcal H)$, then $LS(\mathcal M)=\mathcal M$ and the
local measure topology topology $t(\mathcal M)$
coincides with the uniform topology $\|\cdot\|_{B(\mathcal H)}$.

The following lemma will be used repeatedly.

\begin{lemma}\cite{M. Muratov1}\label{201}
Suppose that $\mathcal M$ is a von Neumann algebra without direct summand of type
$\mathrm{II}$. For every $A$ in $LS(\mathcal M)$,
there exists a sequence $\{Z_i\}$ of mutually orthogonal central projections in $\mathcal{M}$
with $\sum_{i=1}^{\infty}Z_i=I$, such that $A=\sum_{i=1}^{\infty}Z_iA$ and $Z_iA\in\mathcal{M}$ for every $i$.
\end{lemma}

\section{Centralizable mappings on algebras of locally measurable operators}\

The following theorem is the main result in this paper.

\begin{theorem}\label{301}
Suppose that $\mathcal M$ is a von Neumann algebra without direct summands of type $\mathrm{I}_1$ and type
$\mathrm{II}$, $\mathcal A$ is a $*$-subalgebra of $LS(\mathcal{M})$ containing $\mathcal M$.
Then every element $G$ in $\mathcal{A}$ is a full-centralizable point of $L_{t(\mathcal M)}(\mathcal A,LS(\mathcal M))$.
\end{theorem}

To prove Theorem 3.1, we need the following lemmas.

\begin{lemma}\label{302}
Let $\mathcal{M}$ be a von Neumann algebra and $Z$ be a central projection in
$P(\mathcal Z(\mathcal M))$.
Then we have that $LS(Z\mathcal{M})=ZLS(\mathcal{M})$.
\end{lemma}

\begin{proof}
Since the unit element of $LS(Z\mathcal{M})$ is $Z$,
we have that
$$LS(Z\mathcal{M})=ZLS(Z\mathcal{M})\subseteq ZLS(\mathcal{M}).$$
Similarly, we can obtain that $LS((I-Z)\mathcal{M})\subseteq (I-Z)LS(\mathcal{M})$.
Clearly, $\mathcal{M}=Z\mathcal{M}\oplus(I-Z)\mathcal{M}$,
it follows that
$$ZLS(\mathcal{M})= ZLS(Z\mathcal{M})\oplus ZLS((I-Z)\mathcal{M}).$$
By $LS((I-Z)\mathcal{M})\subseteq(I-Z)LS(\mathcal{M})$, we know that $ZLS((I-Z)\mathcal{M})=0$.
It means that $ZLS(\mathcal{M})=ZLS(Z\mathcal{M})=LS(Z\mathcal{M})$.
\end{proof}

In the following, we always assume that $\mathcal M$ is a von Neumann algebra without direct summands of type $\mathrm{I}_1$ and type
$\mathrm{II}$ on a Hilbert space $\mathcal H$,
$\mathcal A$ is a $*$-subalgebra of $LS(\mathcal{M})$ containing $\mathcal M$.

\begin{lemma}\label{303}
Suppose that $G$ is a fixed element in $\mathcal{A}$ and $\{Q_i\}_{i=1}^{n}$ is a family of mutually orthogonal central projections in $\mathcal M$ with sum $I$.
If $Q_iG$ is a full centralizable point of $L_{t(\mathcal M)}(Q_i\mathcal A,Q_iLS(\mathcal M))$ for every $i\in\overline{1,n}$, then $G$
is a full centralizable point of $L_{t(\mathcal M)}(\mathcal A,LS(\mathcal M))$
\end{lemma}

\begin{proof}
Let $\phi$ be in $L_{t(\mathcal M)}(\mathcal A, LS(\mathcal M))$ centralizable at $G$.

Firstly, we show that $\phi(Q_i\mathcal A)\subseteq Q_iLS(\mathcal M)$.
Let $A_i$ be an invertible element in $Q_i\mathcal M$, and $t$ be an arbitrary nonzero element in $\mathbb{C}$.
It is easy to show that
$$(I-Q_i+t^{-1}GA_i^{-1})((I-Q_i)G+tA_i)=G.$$
Thus we have that
$$(I-Q_i+t^{-1}GA_i^{-1})\phi((I-Q_i)G+tA_i)=\phi(G).$$
Considering the coefficient of $t$, since $t$ is arbitrarily chosen,
we know that $(I-Q_i)\phi(A_i)=0$. It follows that $\phi(A_i)=Q_i\phi(A_i)\in Q_iLS(\mathcal M)$ for every invertible element $A_i$ in $Q_i\mathcal M$.
Clearly, $Q_i\mathcal M$ is a von Neumann algebra
and every element in $Q_i\mathcal M$ can be written into the sum of two invertible elements in $Q_i\mathcal M$.
Hence we have that $\phi(A_i)\in Q_iLS(\mathcal M)$ for every $A_i$ in $Q_i\mathcal M$.

For every $A_i$ in $Q_i\mathcal A$, by Lemma \ref{201}, we know that there exists a sequence $\{A_i^n\}$ in $Q_i\mathcal M$
converging to $A_i$ with respect to the
local measure topology $t(\mathcal M)$.
By Lemma \ref{301}, we know that $Q_iLS(\mathcal M)=LS(Q_i\mathcal M)$,
since $\phi(A_i^n)\in Q_iLS(\mathcal M)=LS(Q_i\mathcal M)$ and $\phi$ is continuous with respect to the
local measure topology $t(\mathcal M)$, we have that $\phi(A_i)\in Q_iLS(\mathcal M)$ for every $A_i$ in $Q_i\mathcal A$.

In the following we show that $\phi$ is a centralizer  from $\mathcal{A}$ into $LS(\mathcal{M})$.

Suppose that $A$ and $B$ are two elements in $\mathcal A$ with $AB=G$.
Since $\{Q_i\}_{i=1}^{n}$ is a family of mutually orthogonal central projections in $\mathcal A$ with sum $I$,
we know that there exist some elements $A_i$, $B_i$ and $G_i$
in $Q_i\mathcal{A}$ with $A=\sum_{i=1}^{n}A_i$, $B=\sum_{i=1}^{n}B_i$ and $G=\sum_{i=1}^{n}G_i$. Moreover, we have that
$A_iB_i=G_i$.

Denote the restriction of $\phi$ in $Q_i\mathcal{A}$ by $\phi_i$.
By $\phi(\mathcal{A}_i)\subseteq Q_iLS(\mathcal{M})$, it implies that
$$\sum_{i=1}^{n}\phi(G_i)=\phi(G)=\phi(A)B=\sum_{i=1}^{n}\phi(A_i)\sum_{i=1}^{n}B_i=\sum_{i=1}^{n}\phi(A_i)B_i.$$
Hence we can obtain that $\phi_i(G_i)=\phi_i(A_i)B_i$.
Similarly, we have that $\phi_i(G_i)=A_i\phi_i(B_i)$.
By assumption, $G_i$ is a full-centralizable point of $L_{t(\mathcal M)}(Q_i\mathcal A,Q_iLS(\mathcal M))$ for every $i\in\overline{1,n}$,
that is $\phi_i$ is a centralizer from $Q_i\mathcal A$ into $Q_iLS(\mathcal M)) $.
It follows that
$$\phi(A)=\sum_{i=1}^{n}\phi_i(A_i)=\sum_{i=1}^{n}\phi_i(Q_i)A_i=\sum_{i=1}^{n}\phi_i(Q_i)\sum_{i=1}^{n}A_i=\phi(I)A.$$
Similarly, we can prove $\phi(A)=A\phi(I)$.
It means that $G$ is a full-centralizable point of $L_{t(\mathcal M)}(\mathcal A,LS(\mathcal M))$.
\end{proof}

For a unital algebra $\mathcal{A}$ and a unital left $\mathcal{A}$-module $\mathcal{M}$, we call an element $A$ in
$\mathcal{A}$ a \emph{right separating point}
of $\mathcal{M}$ if $MA=0$ implies $M=0$ for every $M\in\mathcal{M}$.
It is easy to see that every right invertible element in $\mathcal{A}$ is a right separating point of $\mathcal{M}$.

\begin{lemma}\label{304}
Suppose that $G$ is a fixed element in $\mathcal{A}$.
If $G$ is injective and the range of $G$ is dense in $\mathcal H$,
then $G$ is a full-centralizable point of of $L_{t(\mathcal M)}(\mathcal A,LS(\mathcal M))$.
\end{lemma}

\begin{proof}
Firstly, we show that $G$ is a right separating point of $LS(\mathcal M)$.
Let $A$ be in $LS(\mathcal{M})$ with $AG=0$, by Lemma \ref{201}, we know that
there exists a sequence $\{Z_i\}$ of mutually orthogonal central projections in $\mathcal{M}$
with $\sum_{i=1}^{\infty}Z_i=I$, such that $A=\sum_{i=1}^{\infty}Z_iA$ and $Z_iA\in\mathcal{M}$ for every $i$.

By $AG=0$, we have that $\sum_{i=1}^{\infty}Z_iAG=0$.
Since $\{Z_i\}$ are mutually orthogonal projections,
it follows that $Z_iAG=0$. By the range of $G$ is dense in $\mathcal{H}$
and $Z_iA\in\mathcal{M}$ for every $i$, it is easy to show that
$Z_iA=0$ for every $i$. It means that $A=\sum_{i=1}^{\infty}Z_iA=0$.

Let $\phi$ be in $L_{t(\mathcal M)}(\mathcal A, LS(\mathcal M))$ centralizable at $G$
and $A$ be an invertible element in $\mathcal M$. It follows that
$$\phi(I)G=\phi(G)=\phi(AA^{-1}G)=\phi(A)A^{-1}G.$$
Since $G$ is a right separating point of $LS(\mathcal M)$, we have $\phi(I)=\phi(A)A^{-1}$.
That is $\phi(A)=\phi(I)A$ for every invertible element $A$ in $\mathcal M$.
It follows that $\phi(A)=\phi(I)A$ for every $A$ in $\mathcal M$.
Since $\phi$ is continuous with respect to the
local measure topology $t(\mathcal M)$,
we know that $\phi(A)=\phi(I)A$ for every $A$ in $\mathcal A$.

Similarly, by $ker(G)=\{0\}$, we can obtain that $\phi(A)=A\phi(I)$  for every $A$ in $\mathcal A$.
It means that $\phi$ is a centralizer from $\mathcal{A}$ into $LS(\mathcal M)$.
\end{proof}

\begin{lemma}\label{305}
$G=0$ is a full-centralizable point of $L_{t(\mathcal M)}(\mathcal A,LS(\mathcal M))$.
\end{lemma}

\begin{proof}
Since $\mathcal{M}$ is a von Neumann algebra without direct summand of type $\mathrm{I}_1$,
it is well known that $\mathcal{M}$ is generated algebraically by all
idempotents in $\mathcal{M}$.

Let $\phi$ be in $L_{t(\mathcal M)}(\mathcal A, LS(\mathcal M))$ centralizable at $G$.
Define a bilinear mapping $\varphi$ from $\mathcal{M}\times\mathcal{M}$
into $LS(\mathcal M)$ by
$\varphi(A,B)=\phi(A)B$
for each $A,B$ in $\mathcal{M}$. By assumption we have that $\phi(A)B=0$ for each $A,B$ in $\mathcal{M}$ with $AB=0$.
It follows that $AB=0$ implies $\varphi(A,B)=0$.
By \cite[Theorem~4.1]{M. Bresar 3}, we can obtain that $\varphi(A,I)=\varphi(I,A)$, it implies that
$\phi(A)=\phi(I)A$ for every $A$ in $\mathcal{M}$.
It also holds for every $A$ in $\mathcal A$,
since $\phi$ is continuous with respect to the
local measure topology $t(\mathcal M)$.

Similarly, we can prove that $\phi(A)=A\phi(I)$ for every $A$ in $\mathcal A$.
It means that $\phi$ is a centralizer from $\mathcal{A}$ into $LS(\mathcal M)$.
\end{proof}

Let $A$ be an element in $\mathcal A$. The central carrier $\mathcal{C}(A)$
of $A$ in a von Neumann algebra $\mathcal M$ is the projection $I-P$,
where $P$ is the union of all central projections $P_\alpha$ in $\mathcal Z(\mathcal M)$
such that $P_\alpha A=0$.

\begin{lemma}\label{306}
Suppose that $G$ is a fixed element in $\mathcal{A}$.
If $\mathcal{C}(P)=\mathcal{C}(I-P)=I$, where $P$ is
the range projection of $G$,
then $G$ is a full-centralizable point of $L_{t(\mathcal M)}(\mathcal A,LS(\mathcal M))$.
\end{lemma}

\begin{proof}
Let $P_1=P,~ P_2=I-P$
and denote $P_i\mathcal{A}P_j$ and $P_i LS(\mathcal{M})P_j$ by $\mathcal{A}_{ij}$ and $\mathcal{B}_{ij}$, respectively, $i,j=1,2$.
For every $A$ in $\mathcal{A}$, denote $P_iAP_j$ by $A_{ij}$.

Firstly, we claim that for every element $A$ in $LS(\mathcal{M})$, the condition $A\mathcal{A}_{ij}=0$ implies $AP_i=0$
and  $\mathcal{A}_{ij}A=0$ implies $P_jA=0$.

Indeed, since $\mathcal{C}(P_j)=I$, by \cite[Proposition 5.5.2]{514} and $\mathcal M\subseteq\mathcal A$,
we know that the range of $\mathcal{A}P_j$ is dense in $\mathcal H$.
Thus $AP_i\mathcal{A}P_j=0$ implies $AP_i=0$.
On the other hand, if $\mathcal{A}_{ij}A=0$, then $A^{*}\mathcal{A}_{ji}=0$. Hence $A^{*}P_j=0$ and $P_jA=0$.

Besides, since $P_1=P$ is the range projection of $G$, we have that $P_1G=G$.
Moreover, for every element $A$ in $LS(\mathcal{M})$, $AG=0$ if and only if $AP_1=0$.

Let $\phi$ be in $L_{t(\mathcal M)}(\mathcal A, LS(\mathcal M))$ centralizable at $G$.
In the following, we show that $\phi(\mathcal{A}_{ij})\subseteq\mathcal{B}_{ij}$, respectively, $i,j=1,2$.
Suppose that $A_{11}$ is an invertible element in $\mathcal{A}_{11}$, and
$A_{12}, A_{21}, A_{22}$ are arbitrary elements in $\mathcal{A}_{12},\mathcal{A}_{21}, \mathcal{A}_{22}$, respectively.
Let $t$ be an arbitrary nonzero element in $\mathbb{C}$.

\textbf{Claim 1:} $\phi(\mathcal{A}_{12})\subseteq\mathcal{B}_{12}$.

By $(P_1+tA_{12})G=G$, we have that $\phi(G)=\phi(P_1+tA_{12})G$.
It implies that $\phi(A_{12})G=0$. Hence $\phi(A_{12})P_1=0$.

By $(P_1+tA_{12})G=G$, we have that $\phi(G)=(P_1+tA_{12})\phi(G)$.
It follows that $A_{12}\phi(G)=A_{12}\phi(P_1)G=0$.
Thus $A_{12}\phi(P_1)P_1=0$ and $P_2\phi(P_1)P_1=0$.

By
$(A_{11}+tA_{11}A_{12})(A_{11}^{-1}G-A_{12}A_{22}+t^{-1}A_{22})=G,$
we have that
\begin{align}
\phi(A_{11}+tA_{11}A_{12})(A_{11}^{-1}G-A_{12}A_{22}+t^{-1}A_{22})=\phi(G).\label{203}
\end{align}
Since $t$ is arbitrarily chosen in \eqref{203}, we can obtain that
$$\phi(A_{11})(A_{11}^{-1}G-A_{12}A_{22})+\phi(A_{11}A_{12})A_{22}=\phi(G).$$
Since $A_{12}$ is also arbitrarily chosen,
we can obtain $\phi(A_{11})A_{12}A_{22}=\phi(A_{11}A_{12})A_{22}.$  Taking $A_{22}=P_2$, since $\phi(A_{12})P_1=0$, we have
\begin{align}
\phi(A_{11}A_{12})=\phi(A_{11})A_{12}.\label{205}
\end{align}
Taking $A_{11}=P_1$, by $P_2\phi(P_1)P_1=0,$ it implies that
\begin{align}
P_2\phi(A_{12})=P_2\phi(P_1)A_{12}=0.\label{207}
\end{align}
Thus we can obtain that
$$\phi(A_{12})=\phi(A_{12})P_1+P_1\phi(A_{12})P_2+P_2\phi(A_{12})P_2
=P_1\phi(A_{12})P_2\subseteq\mathcal{B}_{12}. $$

\textbf{Claim 2} $\phi(\mathcal{A}_{11})\subseteq\mathcal{B}_{11}$.

Considering the coefficient of $t^{-1}$ in \eqref{203}, we have that $\phi(A_{11})A_{22}=0.$ Thus
$\phi(A_{11})P_2=0.$
By \eqref{205}, we obtain that $P_2\phi(A_{11})A_{12}=P_2\phi(A_{11}A_{12})=0.$ It follows that
$P_2\phi(A_{11})P_1=0.$
Therefore, $\phi(A_{11})=P_1\phi(A_{11})P_1\subseteq\mathcal{B}_{11}$
for every invertible element $A_{11}$ in $\mathcal A_{11}$.
Since $\phi$ is continuous with respect to the
local measure topology $t(\mathcal M)$,
it implies that $\phi(A_{11})\subseteq\mathcal{B}_{11}$
for every $A_{11}$ in $\mathcal A_{11}$.

\textbf{Claim 3} $\phi(\mathcal{A}_{22})\subseteq\mathcal{B}_{22}$.

By
$(A_{11}+tA_{11}A_{12})(A_{11}^{-1}G-A_{12}A_{22}+t^{-1}A_{22})=G,$
we can obtain that
$$(A_{11}+tA_{11}A_{12})\phi(A_{11}^{-1}G-A_{12}A_{22}+t^{-1}A_{22})=\phi(G).$$
Through a similar discussion for equation \eqref{203}, we can show that
\begin{align}
P_1\phi(A_{22})=0~\mathrm{and}~\phi(A_{12}A_{22})=A_{12}\phi(A_{22}).\label{211}
\end{align}
Thus $A_{12}\phi(A_{22})P_1=\phi(A_{12}A_{22})P_1=0$.
It follows that
$P_2\phi(A_{22})P_1=0.$
Therefore, $\phi(A_{22})=P_2\phi(A_{22})P_2\subseteq\mathcal{B}_{22}. $

\textbf{Claim 4} $\phi(\mathcal{A}_{21})\subseteq\mathcal{B}_{21}$.

By
$(A_{11}+tA_{11}A_{12})(A_{11}^{-1}G-A_{12}A_{21}+t^{-1}A_{21})=G,$
we have that
$$(A_{11}+tA_{11}A_{12})\phi(A_{11}^{-1}G-A_{12}A_{21}+t^{-1}A_{21})=\phi(G).$$
According to this equation, we can similarly obtain that
$P_1\phi(A_{21})=0$
and
\begin{align}
A_{12}\phi(A_{21})=\phi(A_{12}A_{21}).\label{219}
\end{align}
Hence $A_{12}\phi(A_{21})P_2=\phi(A_{12}A_{21})P_2=0$.
It follows that $P_2\phi(A_{21})P_2=0$. Therefore, $\phi(\mathcal{A}_{21})=P_2\phi(A_{21})P_1\subseteq\mathcal{B}_{21}$.

\textbf{Claim 5} $\phi(A_{ij})=\phi(P_i)A_{ij}=A_{ij}\phi(P_j)$ for each $i,j\in \{1,2\}$.

By taking $A_{11}=P_1$ in \eqref{205}, we have that
$\phi(A_{12})=\phi(P_1)A_{12}$.
By taking $A_{22}=P_2$ in \eqref{211}, we have that
$\phi(A_{12})=A_{12}\phi(P_2)$.

By \eqref{205}, we have $\phi(A_{11})A_{12}=\phi(A_{11}A_{12})=\phi(P_1)A_{11}A_{12}$. It follows that
$\phi(A_{11})=\phi(P_1)A_{11}.$ On the other hand, $\phi(A_{11})A_{12}=\phi(A_{11}A_{12})=A_{11}A_{12}\phi(P_2)=A_{11}\phi(A_{12})=A_{11}\phi(P_1)A_{12}$.
It follows that $\phi(A_{11})=A_{11}\phi(P_1)$ for every invertible element $A_{11}$ and so for all elements in $\mathcal A_{11}$.

By \eqref{211} and \eqref{219}, through a similar discussion as above, we can obtain that
$\phi(A_{22})=A_{22}\phi(P_2)=\phi(P_2)A_{22}$
and
$\phi(A_{21})=A_{21}\phi(P_1)=\phi(P_2)A_{21}.$

Now we have proved that
$\phi(\mathcal{A}_{ij})\subseteq\mathcal{B}_{ij}$
and $\phi(A_{ij})=\phi(P_i)A_{ij}=A_{ij}\phi(P_j).$
It follows that
\begin{align*}
\phi(A)&=\phi(A_{11}+A_{12}+A_{21}+A_{22})\notag\\
&=\phi(P_1)(A_{11}+A_{12}+A_{21}+A_{22})+\phi(P_2)(A_{11}+A_{12}+A_{21}+A_{22})\notag\\
&=\phi(P_1+P_2)(A_{11}+A_{12}+A_{21}+A_{22})\notag\\
&=\phi(I)A.\notag
\end{align*}
Similarly, we can prove that $\phi(A)=A\phi(I)$.
\end{proof}

In the following, we give the proof of our main result.

\begin{proof}[Proof of the Theorem 3.1.]
Let $Q_1=I-\mathcal{C}(I-P)$, $Q_2=I-\mathcal{C}(P)$, and $Q_3=I-Q_1-Q_2$,
where $P$ is the range projection of $G$.
Obviously, $Q_1\leq P$ and $Q_2\leq I-P$, it follows that $\{Q_i\}_{i=1,2,3}$ are mutually orthogonal central projections with sum $I$.
Thus we have that
$\mathcal A=\sum\limits_{i=1}^{3}(Q_i\mathcal{A}).$
Denote $Q_i\mathcal{A}$ by $\mathcal{A}_i$.
For every element $A$ in $\mathcal{A}$, we can write $A=\sum\limits_{i=1}^{3}A_i=\sum\limits_{i=1}^{3}Q_iA$.

Next we divide the proof into two cases.

\textbf{Case 1:} Suppose that $G$ is injective, that is $ker(G)=\{0\}$.

Since $Q_1\leq P$, we have that $\overline{ran G_1}=\overline{ran Q_1G}=Q_1\mathcal H$.
By assumption we know that $G_1=Q_1G$ is injective on $Q_1\mathcal H$.
By Lemma \ref{304}, we know that $G_1$ is a full-centralizable point of $L_{t(\mathcal M)}(\mathcal A_1,LS(Q_1\mathcal M))$.

Since $Q_2\leq I-P$, we have that $G_2=Q_2G=0$.
By Lemma \ref{305}, we know that $G_2$ is a full-centralizable point of $L_{t(\mathcal M)}(\mathcal A_2,LS(Q_2\mathcal M))$.

Since $P$ is the range projection of $G$, it follows that $\overline{ran G_3}=\overline{ran Q_3G}=Q_3P=P_3$.
Denote the central carrier of $P_3$ in $\mathcal{A}_3$ by $\mathcal{C}_{\mathcal{A}_3}(P_3)$. We have that
$$Q_3-\mathcal{C}_{\mathcal{A}_3}(P_3)\leq Q_3-P_3=Q_3(I-P)\leq I-P.$$
Obviously, $Q_3-\mathcal{C}_{\mathcal{A}_3}(P_3)$ is a central projection orthogonal to $Q_2$.
Thus
$$Q_3-\mathcal{C}_{\mathcal{A}_3}(P_3)+I-\mathcal{C}(P)\leq I-P.$$
It implies that $Q_3-\mathcal{C}_{\mathcal{A}_3}(P_3)+P\leq \mathcal{C}(P)$.
Hence we have that $Q_3-\mathcal{C}_{\mathcal{A}_3}(P_3)=0$,
that is $\mathcal{C}_{\mathcal{A}_3}(P_3)=Q_3$. Similarly, we can show that
$\mathcal{C}_{\mathcal{A}_3}(Q_3-P_3)=Q_3$. By Lemma \ref{306},
we know that $G_3$ is a full-centralizable point of $L_{t(\mathcal M)}(\mathcal A_3,LS(Q_3\mathcal M))$.

By Lemma \ref{302}, we can obtain that $Q_iLS(\mathcal{M})=LS(Q_i\mathcal{M})$.
Hence $G_i$ is a full-centralizable point of $L_{t(\mathcal M)}(\mathcal A_i,Q_i LS(\mathcal M))$
for each $i=1,2,3$.

By Lemma \ref{303}, it follows that $G$ is a full-centralizable point of $L_{t(\mathcal M)}(\mathcal A,LS(\mathcal M))$.

\textbf{Case 2:} Suppose that $ker(G)\neq \{0\}$.

In this case, $G_2$ and $G_3$ are still full-centralizable points of
$L_{t(\mathcal M)}(\mathcal A_2,LS(Q_2\mathcal M))$ and $L_{t(\mathcal M)}(\mathcal A_3,LS(Q_3\mathcal M))$,
respectively.

Since $\overline{ran G_1}=Q_1\mathcal{H}$ , we have that $ker(G_1^{*})=\{0\}$.
By Case 1, we know that $G_1^{*}$ is a full-centralizable point of $L_{t(\mathcal M)}(\mathcal A_1,LS(Q_1\mathcal M))$.
Next we show that $G_1$ is also a full-centralizable point of $L_{t(\mathcal M)}(\mathcal A_1,LS(Q_1\mathcal M))$.

In fact, let $\phi_1$ be in $L_{t(\mathcal M)}(\mathcal A_1,LS(Q_1\mathcal M))$ centralizable at $G_1$.
Define a linear mapping $\widetilde{\phi_1}$ from $\mathcal{A}_1$ into $LS(Q_1\mathcal M)$
by $\widetilde{\phi_1}(A)=(\phi_1(A^*))^*$ for every $A$ in $\mathcal{A}_1$.
Suppose that $A$ and $B$ are two elements in $\mathcal A_1$ with $AB=G_1$, we have that $B^*A^*={G_1}^*$.
It follows that
$$\phi_1(G)=A\phi_1(B)=\phi_1(A)B.$$
By the definition of $\widetilde{\phi_1}$, we can obtain that
$$\widetilde{\phi_1}(G^*)=B^*\widetilde{\phi_1}(A^*)=\widetilde{\phi_1}(B^*)A^*.$$
Since $G^{*}$ is a full-centralizable point of $L_{t(\mathcal M)}(\mathcal A_1,LS(Q_1\mathcal M))$,
we have that $\widetilde{\phi_1}$ is a centralizer.
Thus $\phi_1$ is also a centralizer. It means that $G_1$ is a full-centralizable point of $L_{t(\mathcal M)}(\mathcal A_1,LS(Q_1\mathcal M))$.

By Lemma \ref{303}, we know $G$ is a full-centralizable point of $L_{t(\mathcal M)}(\mathcal A,LS(Q\mathcal M))$.
\end{proof}

\end{document}